\documentclass[12pt]{amsart}

\usepackage{amscd,amsthm}
\usepackage{latexsym}
\usepackage{amsmath,amstext,amsthm,amsfonts,amssymb}
\usepackage[dvips]{graphicx}
\usepackage{anysize}
\marginsize{22mm}{18mm}{17mm}{30mm}
\usepackage{newlfont}
\usepackage[ansinew]{inputenc}
\usepackage{graphpap}

\usepackage{mathrsfs}

\usepackage[active]{srcltx}
\usepackage[colorlinks=true,linkcolor=red,urlcolor=black]{hyperref}

\theoremstyle{plain} 
\newtheorem{teo}{Teorema }[section]
\newtheorem{maintheorem}{Theorem}

\newtheorem{lema}[teo]{Lemma}
\newtheorem{propo}[teo]{Proposition}

\newtheorem{defi}[teo]{Definition}

\newtheorem{cas2}{Case }
\newtheorem{cas3}{Case }
\title{Topological Entropy for Discontinuous Semiflows}

\author{Nelda Jaque}
\address{
Nelda Jaque, 
Departamento de Matem\'aticas, 
Universidad de Atacama, 
Copayapu 485, Casilla 240,
Copiap\'o, Chile.}
\email{njaquetamblay@gmail.com}

\author{Bernado San Mart\'in}
\address{
Bernardo San Mart\'in,
Departamento de Matem\'aticas, 
Universidad Cat\'olica del Norte, 
Av. Angamos 0610, Casilla 1280, Antofagasta, Chile.}
\email{sanmarti@ucn.cl}
\thanks{This project was partially funded by the Chilean FONDECYT grant 1181183. The first author was also supported by the Chilean CONICYT grant 21160656.}

\begin{document}

\maketitle

\begin{abstract}
We study two variations of Bowen's definitions of topological entropy based on separated and spanning sets which can be applied to the study of discontinuous semiflows on compact metric spaces. We prove that these definitions reduce to Bowen's ones in the case of continuous semiflows. As a second result, we prove that our entropies give a lower bound for the $\tau$-entropy defined by Alves, Carvalho and V\'asquez (2015). Finally, we prove that  for impulsive semiflows satisfying certain regularity condition, there exists a continuous semiflow defined on another compact metric space which is related to the first one by a semiconjugation, and whose topological entropy equals our extended notion of topological entropy by using separated sets for the original semiflow. 

\end{abstract}

\maketitle

\section{Introduction}
Entropy is a notion that quantifies the complexity of a dynamical system. 
This notion was introduced into ergodic theory by Kolgomorov in 1958 and by Sinai in 1959 (see \cite{kolmogorov} and \cite{sinai}, respectively).   
 The entropy introduced by them is usually called metric entropy. 
Since then, this notion has played an important role in the classification of dynamical systems.  
The concept of entropy for  dynamical systems on compact topological spaces  
was given by Adler, Konheim and McAndrew in 1965  (see \cite{alder}). They defined entropy for continuous maps 
in a purely topological way, by using the concept of open cover.  
In 1971,  Bowen introduced  two new definitions of entropy for dynamical systems on  metric spaces, the first one by using the notion of spanning sets and the second one by using the notion of separated sets, and proved that both definitions agree when the dynamical system is continuous and the space is compact (see \cite{bowen}). Moreover, these definitions agree with the topological entropy introduced by Adler et al. The relationship between metric entropy and topological entropy is given by the well-known variational principle.

In this paper we will focus specifically on not necessarily continuous semiflows defined on compact metric spaces. It turns out that the entropies defined by Bowen do not work well in this context. Indeed, it is easy to construct simple examples of discontinuous semiflows where these entropies are infinite (see Section \ref{6}). Motivated by this problem, one can try to find variations of Bowen's definitions of topological entropies that can be applied to the study of not necessarily continuous semiflows. In this direction, Alves, Carvalho and V\'asquez (\cite{v2}) have introduced the notion of topological $\tau$-entropy, which depends on an admissible function $\tau$. Their definition only makes use of separated sets. In this work we study another variation of Bowen's original definition of topological entropies, making use of separated and spanning sets. As a first result we prove that these definitions agree with the usual ones in the case of continuous semiflows. As a second result, we prove that our notions of entropy  give a lower bound for the $\tau$-entropy defined by Alves et.~al, for any admissible function $\tau$. For the rest of the paper we focus on the study of impulsive semiflows satisfying certain regularity conditions. Important contributions on this topic have been made by Kaul (see \cite{i0}), Ciescielski (see \cite{i5} and \cite{i9}), Bonotto (see \cite{i7}) and Bonotto, Bortolan, Caraballo and Collegari (see \cite{i8}). As our third result, we prove that for regular impulsive semiflows, there exists a continuous semiflow on another compact metric space which is related to the first semiflow by a semiconjugation, and whose topological entropy equals our extended notion of topological entropy by using separated sets for the first semiflow.

\section{Setting and Statements}

Here, and throughout this paper, we  denote by  $(X,d)$ a compact metric space and $$\phi:\mathbb{R}^+_0\times X\to X$$ a semiflow that is not necessarily continuous. We will use  the notation $\phi(t,x) = \phi_t(x)$. 
For $\delta>0$, we define the pseudosemimetric\footnote{A pseudosemimetric on $X$ is a map $D:X\times X\to \mathbb{R}_0^+$ satisfying $D(x,x)=0$ and $D(x,y)=D(y,x)$ for all $x,y\in X$.}
$d_{\delta}^{\phi}:X\times X\to\mathbb{R}_0^+$  as
\begin{eqnarray*}\label{parametric}
d_{\delta}^{\phi}(x,y)=\inf\{d(\phi_s(x),\phi_s(y)):s\in[0,\delta)\}.
\end{eqnarray*}
For $x\in X$, $\varepsilon>0$ and $T>0$, let us define
$$\hat{B}(x,\phi,T,\varepsilon,\delta)=\{y\in X : d^{\phi}_{\delta}(\phi_t(x),\phi_t(y))<\varepsilon \mbox{ for all } t\in[0,T]\}.$$
A  set $F\subseteq X$ is said to be\emph{ $(\phi,T,\varepsilon,\delta)$-spanning} if
$$X\subset \bigcup_{y\in F}\hat{B}(y,\phi,T,\varepsilon,\delta),$$
and a set $E\subseteq X$ is said to be\emph{ $(\phi,T,\varepsilon,\delta)$-separated} if 
$$\mbox{for all }x\in E\mbox{ we have } E\cap \hat{B}(x,\phi, T,\varepsilon,\delta)=\{x\}.$$
Define 
$$\hat{r}(\phi,T,\varepsilon,\delta)=\inf\{|F|: F \mbox{ is } (\phi,T,\varepsilon,\delta)\mbox{-spanning}\}$$
and
$$\hat{s}(\phi,T,\varepsilon,\delta)=\sup\{|E|: E \mbox{ is } (\phi,T,\varepsilon,\delta)\mbox{-separated}\}.$$
Here, $|A|$ denotes the cardinality of a subset $A\subseteq X$. Note that $\hat{r}(\phi,T,\varepsilon,\delta)$ or $\hat{s}(\phi,T,\varepsilon,\delta)$ could be infinite.
The \emph{growth rate} of $\hat{r}(\phi,T,\varepsilon,\delta)$ and $\hat{s}(\phi,T,\varepsilon,\delta)$ is defined,  respectively, as
$$\hat{h}_{\text{r}}(\phi,\varepsilon,\delta)=\limsup_{T\to+\infty}\frac{1}{T}\log \hat{r}(\phi,T,\varepsilon,\delta)$$
and
$$\hat{h}_{\text{s}}(\phi,\varepsilon,\delta)=\limsup_{T\to+\infty}\frac{1}{T}\log \hat{s}(\phi,T,\varepsilon,\delta),$$
where we put $\log \infty= \infty$. It is easy to see that the functions $\varepsilon \to \hat{h}_{\text{r}}(\phi,\varepsilon,\delta)$ and $\varepsilon \to \hat{h}_{\text{s}}(\phi,\varepsilon,\delta)$  are nonincreasing. Now, we  define
$$
\hat{h}_{\text{r}}(\phi,\delta)=\lim_{\varepsilon\to0^+}\hat{h}_{\text{r}}(\phi,\varepsilon,\delta) \,\,\,\, \text{and} \,\,\,\,
\hat{h}_{\text{s}}(\phi,\delta)=\lim_{\varepsilon\to0^+}\hat{h}_{\text{s}}(\phi,\varepsilon,\delta),$$
where $\hat{h}_{\text{r}}(\phi,\delta)$ or $\hat{h}_{\text{s}}(\phi,\delta)$ could be infinite.
Here, the functions $\delta \to \hat{h}_{\text{r}}(\phi,\delta)$ and $\delta \to \hat{h}_{\text{s}}(\phi,\delta)$  are nonincreasing. We define 
\begin{eqnarray*}
\hat{h}_{\text{r}}(\phi)=\lim_{\delta\to0^+}\hat{h}_{\text{r}}(\phi,\delta) \,\,\,\, \text{and} \,\,\,\, \hat{h}_{\text{s}}(\phi)=\lim_{\delta\to0^+}\hat{h}_{\text{s}}(\phi,\delta).
\end{eqnarray*}

When $\phi$ is continuous, we recover Bowen's classical entropy that we denote by $h_{top}(\phi)$ (for the definition see Section \ref{3}). This is the content of our first theorem.
\begin{maintheorem}\label{A}
\begin{upshape}
Let $\phi:\mathbb{R}_0^+\times X\to X$ be a continuous semiflow on the compact metric space $(X,d)$. Then
$$\hat{h}_{\text{r}}(\phi)=\hat{h}_{\text{s}}(\phi)=h_{top}(\phi).$$
\end{upshape}
\end{maintheorem}

As mentioned in the Introduction, some variants of the notion of entropy for not necessarily continuous semiflows have already been studied before. In particular, in 2015, Alves, Carvalho and V\'asquez (\cite{v2})  introduced the concept of topological $\tau$-entropy for this type of semiflows, where $\tau$ is a so-called admissible function. Their definitions are as follows. Let $\tau$ be a function assigning to each $x\in X$ a strictly increasing sequence $(\tau_n(x))_{n\in A_0(x)}$ of nonnegative numbers, where either $A_0(x)=\{0,1,...,l\}$ for some
$l\in\mathbb{N}$ or $A_0(x)=\mathbb{N}_0$, and such that $\tau_0(x)=0$ for all $x\in X$. One says that $\tau$ is \emph{admissible} with respect to a subset  $Z\subset X$ if there exists a constant $\gamma>0$ such that
\begin{enumerate}
\item $\tau_1(x)\geq \gamma$ for all $x\in Z$,\end{enumerate}
and for all $x\in X$ and all $n\in \mathbb{N}$ with $n+1\in A(x)$, we have
\begin{enumerate}
\item[(2)] $\tau_{n+1}(x)-\tau_n(x)\geq\gamma$,
\item[(3)] $\tau_1(\phi_s(x))=\tau_{n}(x)-s$ if $\tau_{n-1}(x)<s<\tau_n(x)$, and
\item[(4)] $\tau_1(\phi_s(x))=\tau_{n+1}(x)$ if $s=\tau_n(x)$.
\end{enumerate}
For each admissible function $\tau$, $x\in X$, $T>0$ and $\rho>0$ with $\rho<\gamma/2$, one defines
$$J_{T,\rho}^{\tau}(x) = (0,T]\setminus\bigcup_{j\in A_0(x)}(\tau_j(x)-\rho,\tau_j(x)+\rho).$$
The \emph{$\tau$-dynamical ball} of radius $\epsilon>0$  centred at $x$ is defined as
$$B^{\tau}(x,\phi, T,\epsilon,\rho)=\left\{y\in X:d(\phi_t (x),\phi_t (y))< \epsilon,\mbox{ for all }t\in J_{T,\rho}^{\tau}(x) \right\}.$$
Accordingly, a finite set $E\subseteq X$ is said to be $(\phi,\tau, T, \epsilon,\rho)$-\emph{separated} if
$$\mbox{ for all }x\in E \mbox{ we have } E\cap B^{\tau}(x,\phi, T,\epsilon,\rho)=\{x\}.$$
The \emph{$\tau$-topological entropy} of $\phi$ is defined as
$$h^{\tau}_{\text{top}}(\phi)= \lim_{\rho\to 0^+}\lim_{\epsilon\to 0^+}\limsup_{T\to +\infty}\frac{1}{T} \log s^{\tau}(\phi, T,\epsilon,\rho),$$
where
$$
s^{\tau}(\phi, T,\epsilon,\rho)=\sup\,\{|E|:E \text{ is } (\phi,\tau, T, \epsilon,\rho)\text{-separated}\}.$$

The second result of this paper gives a precise comparison between $\hat{h}_{\text{s}}(\phi)$ and $h_{\text{top}}^{\tau}(\phi)$, valid for any $\phi$. It also gives a comparison between $\hat{h}_{\text{s}}(\phi)$ and $\hat{h}_{\text{r}}(\phi)$. More precisely, we prove the following result.
\begin{maintheorem}\label{B}
Let $\phi:\mathbb{R}_0^+\times X\to X$ be a semiflow and $\tau$ an admissible function on $X$. Then
\begin{eqnarray*}\label{C1}
\hat{h}_{\text{r}}(\phi)\leq \hat{h}_{\text{s}}(\phi) \leq h_{\text{top}}^{\tau}(\phi).
\end{eqnarray*}
\end{maintheorem}

The third and last result of this paper relates  $\hat{h}_{\text{s}}(\phi)$, for an impulsive semiflow $\phi$ satisfying some regularity conditions, to the topological entropy $h_{top}(\tilde{\phi})$ of certain continuous semiflow $\tilde{\phi}$  defined on a compact metric space which is related to $\phi$ by a semiconjugation. In order to state our result, we must first give some definitions. Let   $$\varphi:\mathbb{R}^+_0\times X\to X$$ be a continuous semiflow on $X$, $D\subset X$ a proper closed subset  and  $I:D\to X$ a continuous function. The \emph{first impulse time function} $\tau_1:X\to\mathbb{R}_0^+\cup\{\infty\}$  is defined by
$$\tau_1(x)=\left\{\begin{array}{ll}\inf\left\{t> 0:\varphi_t(x)\in D\right\} ,& \text{if } \varphi_t(x)\in D\text{ for some }t>0;\\
                                                                       +\infty, & \text{otherwise.}\end{array}\right.$$
If $\tau_1(x)<\infty$, we define the \emph{first impulse point} as
$$x^1=x^1(x)=\varphi_{\tau_1(x)}(x).$$
Inductively, if $\tau_n(x)<\infty$,  the \emph{$(n+1)$-th impulsive time} is defined by
$$\tau_{n+1}(x)= \tau_1(x) +\sum_{k=1}^{n}\tau_1(I(x^k)),$$
and if $\tau_1(I(x^n))<\infty$, the $(n+1)$-th impulsive point is defined by 
$$x^{n+1}=x^{n+1}(x)=\varphi_{\tau_1(I(x^n))}(I(x^n)).$$
Let $T(x)= \sup \{\tau_n(x):n\geq 1\}$. 
The \emph{impulsive drift} $\gamma_x:[0,T(x))\to X$ for a point $x\in X$  is defined inductively 
by
$$\gamma_x(t)=\varphi_t(x), \mbox{ if } t \in [0, \tau_1(x)),$$
and
$$\gamma_x(t)=\varphi_{t-\tau_n(x)}(I(x^n)), \mbox{ if }  t \in [\tau_n(x), \tau_{n+1}(x)) \mbox{ and }\tau_n(x)<\infty.$$
Observe that if $I(D)\cap D=\emptyset$ then $T(x)=\infty$. We say that $(X,\varphi, D, I)$ is an \emph{impulsive dynamical system} if for all $x\in X$ we have:
 \begin{enumerate}
 \item $\tau_1(x)>0$ and
 \item $T(x)=\infty$.
 \end{enumerate}
For each impulsive dynamical system $(X,\varphi,D,I)$, we define the \emph{associated impulsive semiflow} $\phi: \mathbb{R}_0^+\times X \rightarrow X$  as
$$ \phi_t(x) = \gamma_x(t),$$
where $\gamma_x(t)$ is the impulsive drift for $x\in X$. It is easy to see that $\phi$ is indeed a semiflow (see \cite{i7}), although it is not necessarily continuous. Moreover, when $(X,\varphi, D, I)$ is an impulsive dynamical system, then $\tau_1$ is lower semicontinuous on $X\setminus D$ (see \cite{i5}). \\
For $\eta>0$, we put
$$\begin{array}{rcl}D_{\eta}=\bigcup_{x\in D}\{\varphi_t(x): 0<t<\eta\} &\mbox{and}& X_{\eta}=X\setminus (D\cup D_{\eta}).\end{array}$$
\begin{defi}\begin{upshape}\label{RIDS}We say that the impulsive dynamical system $(X,\varphi, D, I)$ is \emph{regular} if $I(D)\cap D=\emptyset$, $I$ is Lipschitz and if there exists $\eta>0$ such that
\begin{enumerate}
\item $D_{\xi}$ is open for some  $0<\xi<\eta/4$,
\item $\varphi_{\xi}(D_{\xi})\subset X_{\xi}$, and
\item for all $x\in I(D)$ and $t\in(0,\xi]$, $\varphi_t(x)\notin I(D)$.
\end{enumerate}\end{upshape}\end{defi}
The semiflow $\phi: \mathbb{R}_0^+\times X \rightarrow X$ associated to the regular impulsive dynamical system  $(X,\varphi, D, I)$ is called a \emph{regular impulsive semiflow}. The third result of this paper is the following.
\begin{maintheorem}\label{F}
\begin{upshape}
Let $\phi:\mathbb{R}_0^+\times X\to X$ be the semiflow of a regular impulsive dynamical system $(X,\varphi,D,I)$.
Then there exist a compact metric space $Y$, a continuous semiflow $\tilde{\phi}:\mathbb{R}_0^+\times Y\to Y$ and a uniform continuous bijection $H:X_{\xi}\to Y$ such that for all $t\geq0$
$$\tilde{\phi}_t\circ H=H\circ\phi_t.$$
Moreover,
$$h_{top}(\tilde{\phi})=\hat{h}_{\text{s}}(\phi).$$
\end{upshape}
\end{maintheorem}

The rest of the paper is devoted to the proofs of Theorems \ref{A}, \ref{B} and  \ref{F}. These are given in Section \ref{3}, \ref{4}, and \ref{5}, respectively. Finally, in the Section \ref{6}, we give an example of a discontinuous semiflow $\phi$ where the usual notions of entropy are infinite, and calculate the new entropies $\hat{h}_{\text{s}}(\phi)$ and $\hat{h}_{\text{r}}(\phi)$.

\section{Proof of Theorem \ref{A}}\label{3}

Let us consider $\phi:\mathbb{R}_0^+\times X\to X$ a continuous semiflow on the compact metric space $(X,d)$. Before proving Theorem \ref{A}, we briefly recall Bowen's definition of topological entropy. Given $x\in X$, $T>0$ and $\varepsilon>0$ put
$$B(x,\phi,T,\varepsilon)=\{y\in X : d(\phi_t(x),\phi_t(y))<\varepsilon, \mbox{ for all } t\in[0,T]\}.$$
Using this, one defines $(\phi,T,\varepsilon)$-spanning and $(\phi,T,\varepsilon)$-separated sets in the usual way, as done in the Introduction, replacing $B(x,\phi,T,\varepsilon,\delta)$ by $B(x,\phi,T,\varepsilon)$. Next, one defines $r(\phi,T,\varepsilon)$ and $s(\phi,T,\varepsilon)$ accordingly, and puts
$$h_\text{r}(\phi)=\lim_{\varepsilon\to0^+}\limsup_{T\to+\infty}\frac{1}{T}\log r(\phi,T,\varepsilon)$$
and
$$h_\text{s}(\phi)=\lim_{\varepsilon\to0^+}\limsup_{T\to+\infty}\frac{1}{T}\log s(\phi,T,\varepsilon).$$
In \cite{bowen}, Bowen showed that when $X$  compact and $\phi$ is continuous, one has
$$h_{\text{r}}(\phi) = h_{\text{s}}(\phi).$$
He then defined topological entropy for a continuous semiflow as $$h_{top}(\phi)=h_{\text{r}}(\phi) = h_{\text{s}}(\phi).$$
Theorem \ref{A} will follow from the following two lemmas.

\begin{lema}\label{A1}
Let $\phi:\mathbb{R}_0^+\times X\to X$ be a semiflow (not necessarily continuous) on the compact metric space $(X,d)$. Then
\begin{eqnarray*}\label{111}
 \hat{h}_{\text{s}}(\phi)\leq h_{\text{s}}(\phi) \,\,\,\ \text{and} \,\,\,\ \hat{h}_{\text{r}}(\phi)\leq h_{\text{r}}(\phi).
\end{eqnarray*}

\end{lema}

\begin{proof}
Given $\delta>0$, $x\in X$, $\varepsilon>0$ and $T>0$, we have
$$B(x,\phi,T,\varepsilon)\subseteq \hat{B}(x,\phi,T,\varepsilon,\delta).$$
Indeed, suppose $y\notin \hat{B}(x,\phi,T,\varepsilon,\delta)$. Then, there exists $t\in[0,T]$ such that, for all  $s\in[t,t+\delta]$ we have
$$
d(\phi_s(x),\phi_s(y))\geq \varepsilon.
$$
Hence $y\notin B(x,\phi,T,\varepsilon)$. This implies that if $E$ is $(\phi,T,\varepsilon,\delta)$-separated, then $E$ is $(\phi,T,\varepsilon)$-separated. On the other hand, if $F$ is $(\phi,T,\varepsilon)$-spanning, then $F$ is $(\phi,T,\varepsilon,\delta)$-spanning. Therefore
$$ \hat{s}(\phi,T,\varepsilon,\delta)\leq s(\phi,T,\varepsilon)$$
and
$$ \hat{r}(\phi,T,\varepsilon,\delta)\leq r(\phi,T,\varepsilon).$$
Thus
$$ \frac{1}{T}\displaystyle\log \hat{s}(\phi,T,\varepsilon,\delta)\leq \frac{1}{T}\displaystyle\log s(\phi,T,\varepsilon)$$
and
$$\frac{1}{T}\displaystyle\log \hat{r}(\phi,T,\varepsilon,\delta)\leq \frac{1}{T}\displaystyle\log r(\phi,T,\varepsilon).$$
Taking limits we obtain the desired inequalities. 
\end{proof}

\begin{lema}\label{A2}
Let $\phi:\mathbb{R}_0^+\times X\to X$ be a continuous semiflow on the compact metric space $(X,d)$. Then
\begin{eqnarray*}\label{1}
\hat{h}_{\text{s}}(\phi)\geq h_{\text{s}} (\phi) \,\,\,\ \text{and} \,\,\,\ \hat{h}_{\text{r}}(\phi)\geq h_{\text{r}}(\phi).\end{eqnarray*}
\end{lema}
\begin{proof}
Since $\phi$ is a continuous  and $X$ is a compact metric space, for all $\alpha>0$ there exists $\beta=\beta(\alpha)>0$ such that for all $z\in X$ and $t\geq 0$, we have
$$u\in [t,t+\beta]\Rightarrow d(\phi_t(z),\phi_u(z))<\alpha/4.$$
This implies that for any $ x\in X$,  $\delta>0$ with $\delta<\beta$,  $\varepsilon>0$ with $\varepsilon<\alpha/2$ and $T>0$, we have
$$
\hat{B}(x,\phi,T,\varepsilon,\delta)\subseteq B(x,\phi,T,\alpha).
$$
Indeed, if $y\notin B(x,\phi,T,\alpha)$, then there exists $t\in[0,T]$ such that
\begin{eqnarray*}\label{A11}
d(\phi_t(x),\phi_t(y))\geq \alpha.
\end{eqnarray*}
By the triangle inequality, we have  
$$d(\phi_t(x),\phi_t(y))\leq d(\phi_t(x),\phi_u(x))+d(\phi_u(x),\phi_u(y))+d(\phi_u(y),\phi_t(y))$$
for all $u\in(t,t+\beta)$, hence
$$d(\phi_u(x),\phi_u(y))>\alpha/2.$$
Since $\beta>\delta$ and $\varepsilon<\alpha/2$, we have $y\notin \hat{B}(x,\phi,T,\varepsilon,\delta)$. This implies that if $E$ is $(\phi,T,\alpha)$-separated, then $E$ is $(\phi,T,\varepsilon,\delta)$-separated. On the other hand,  if $F$ is $(\phi,T,\varepsilon,\delta)$-spanning, then $F$ is $(\phi,T,\alpha)$-spanning. Therefore
$$\hat{s}(\phi,T,\varepsilon,\delta)\geq s(\phi,T,\alpha)$$
and
$$\hat{r}(\phi,T,\varepsilon,\delta)\geq r(\phi,T,\alpha).$$
Thus
$$\frac{1}{T}\displaystyle\log \hat{s}(\phi,T,\varepsilon,\delta)\geq \frac{1}{T}\displaystyle\log s(\phi,T,\alpha)$$
and
$$\frac{1}{T}\displaystyle\log \hat{r}(\phi,T,\varepsilon,\delta)\geq \frac{1}{T}\displaystyle\log r(\phi,T,\alpha).$$
Taking limits we obtain the desired inequalities.
\end{proof}

\begin{proof}[ Proof of Theorem \ref{A}]
By Bowen  (\cite{bowen}),  we have $h_s(\phi)=h_r(\phi)$ and so, by Lemmas \ref{A1} and \ref{A2}, we get the desired result.
\end{proof}

\section{Proof of Theorem \ref{B}}\label{4}

Theorem \ref{B} will follow from the following two lemmas.

\begin{lema}\label{B2} Let $\phi:\mathbb{R}_0^+\times X\to X$ be a semiflow and $\tau$ an admissible function on $X$. Then
$$\hat{h}_{\text{s}}(\phi)\leq h_{\text{top}}^{\tau}(\phi).$$
\end{lema}
\begin{proof}
Let us consider $\gamma>0$ as in the definition of admissible function. Fix $\rho>0$ with $\rho<\gamma/2$, $\delta>2\rho$, $\varepsilon>0$ and $T>0$. Notice that for all $x\in X$ we have
$$ B^{\tau}(x,\phi,T,\varepsilon, \rho)\subseteq \hat{B}(x,\phi,T,\varepsilon,\delta).$$
Indeed, suppose $y\notin \hat{B}(x,\phi,T,\varepsilon, \delta)$. Then there exists $t\in[0,T]$ such that for all $s\in[t,t+\delta)$
$$d(\phi_{s}(x),\phi_{s}(y))\geq\varepsilon.$$
Since $\delta>2\rho$, we have  $J_{T,\rho}^{\tau}(x)\cap[t,t+\delta]\neq\emptyset$, so there exists $s \in J_{T,\rho}^{\tau}(x)$ such that $d(\phi_s(x), \phi_s(y)) \geq \varepsilon$,  hence $y\notin B^{\tau}(x,\phi,T,\varepsilon, \rho)$. This implies that if  $E$ is $(\phi,T,\varepsilon,\delta)$-separated, then $E$ is $(\phi,\tau,T,\varepsilon,\rho)$-separated. Therefore
$$\hat{s}(\phi,T,\varepsilon,\delta)\leq s^{\tau}(\phi,T,\varepsilon,\rho),$$
hence
$$\frac{1}{T}\log \hat{s}(\phi,T,\varepsilon,\delta)\leq \frac{1}{T}\log s^{\tau}(\phi,T,\varepsilon,\rho).$$
Taking limits, we obtain the desired inequality. 
\end{proof}
\begin{lema}\label{B1} Let $\phi:\mathbb{R}_0^+\times X\to X$ be a semiflow and $\tau$ an admissible function on $X$. Then
$$\hat{h}_{\text{r}}(\phi)\leq \hat{h}_{\text{s}}(\phi).$$
\end{lema}
\begin{proof}
Let us fix $\delta>0$, $\varepsilon >0$ and $T>0$. Since the union of a partially ordered (by set inclusion) family of separated sets is separated, we can take, by Zorn's Lemma, a maximal $(\phi, T,\varepsilon, \delta)$-separated subset $E$. We claim that $E$ is $(\phi, T,\varepsilon, \delta)$-spanning. Indeed, suppose that  $E$ is not $(\phi,T,\varepsilon,\delta)$-spanning. Then there exists $x\in X$ such that for all $y\in E$, with $y\neq x$, there exists $t\in[0,T]$ such that
$$d_{\delta}^{\phi}(\phi_t(x),\phi_t(y))\geq \varepsilon.$$
Therefore $E\cup\{x\}$ is a $(\phi,T,\varepsilon,\delta)$ separated set, which contradicts  the maximality condition of $E$.  This implies
$$\hat{r}(\phi,T,\varepsilon,\delta)\leq |E|\leq \hat{s}(\phi,T,\varepsilon,\delta),$$
hence
$$\frac{1}{T}\log \hat{r}(\phi,T,\varepsilon,\delta)\leq \frac{1}{T}\log \hat{s}(\phi,T,\varepsilon,\delta).$$
Taking limits, we obtain the desired inequality.
\end{proof}
\begin{proof}[Proof of Theorem \ref{B}]
By Lemma \ref{B1} and Lemma \ref{B2}, we have the desired result.
\end{proof}

\section{Proof of Theorem \ref{F}}\label{5}
For the proof of our third result  we need some technical lemmas. 
\begin{lema}\label{semiconj}
\begin{upshape}
Let us consider $\phi$ and $\tilde{\phi}$  two semiflows on the metric spaces $(X,d)$ and $(\tilde{X},\tilde{d})$, respectively. If $H:X\to \tilde{X}$ is a uniformly continuous bijection such that for all $t\geq0$
\begin{eqnarray}\label{semiconjugation}
\tilde{\phi}_t\circ H= H\circ \phi_t,
\end{eqnarray}
then
$$\hat{h}_{\text{s}}(\tilde{\phi})\leq\hat{h}_{\text{s}}(\phi)  \,\,\ \text{and} \,\,\  \hat{h}_{\text{r}}(\tilde{\phi})\leq \hat{h}_{\text{r}}(\phi).$$
\end{upshape}
\end{lema}

\begin{proof}

Let $\varepsilon>0$. Since $H$ is uniformly continuous, there exists $\beta(\varepsilon)=\beta>0$ such that for all $x,y\in X$, we have
$$d(x,y)<\beta \Rightarrow \tilde{d}(H(x),H(y))<\varepsilon.$$
Let us consider $\delta>0$, $T>0$ and  $\tilde{E}\subset\tilde{X}$ a $(\tilde{\phi},T,\varepsilon,\delta)$-separated subset. We claim that $ E= H^{-1}(\tilde{E})$ is $(\phi,T,\beta,\delta)$-separated.
Indeed, given $x\neq y$ in $E$ we have $H(x)\neq H(y)$, hence there exists $t\in[0,T]$ such that
$$\tilde{d}_{\delta}^{\tilde{\phi}}(\tilde{\phi}_t(H(x)),\tilde{\phi}_t(H(y)))\geq\varepsilon.$$ By \eqref{semiconjugation} we have $\tilde{d}_{\delta}^{\phi}(H(\phi_t(x)),H(\phi_t(y)))\geq\varepsilon$,
thus
$$ d_{\delta}^{\phi}(\phi_t(x),\phi_t(y))\geq\beta.$$
This proves that $E$ is $(\phi,T,\beta,\delta)$-separated. Since 
$|\tilde{E}|=|E|$, 
we have
$$\hat{s}(\tilde{\phi},T,\varepsilon,\delta)\leq \hat{s}(\phi,T,\beta,\delta).$$
Taking logarithms and limits (noting that that $\beta\to0^+$ when $\varepsilon\to0^+$) we deduce the first inequality. Now, let
$F\subset X$ be a $(\phi,T,\beta,\delta)$-spanning subset. We claim that $H(F)$ is  $(\tilde{\phi},T,\varepsilon,\delta)$-spanning. Indeed,  for all $\tilde{x}\in \tilde{X}$ there exists $y\in F$ such that for all $t\in[0,T]$ we have 
$d_{\delta}^{\phi}(\phi_t(H^{-1}(\tilde{x})),\phi_t(y))<\beta $. This implies $$d_{\delta}^{\phi}(H(\phi_t(H^{-1}(\tilde{x}))),H(\phi_t(y)))<\varepsilon,$$
and so, by \eqref{semiconjugation}, we deduce
$$\tilde{d}_{\delta}^{\tilde{\phi}}(\tilde{\phi}_t (\tilde{x}),\tilde{\phi}_t( H(y)))<\varepsilon.$$
This proves that $H(F)$ is $(\tilde{\phi},T,\varepsilon,\delta)$-spanning. Since
$|H(F)|=|F| $
we have 
$$ \hat{r}(\tilde{\phi},T,\varepsilon,\delta) \leq \hat{r}(\phi,T,\beta,\delta).$$
Again, taking logarithms and limits (noting that that $\beta\to0^+$ when $\varepsilon\to0^+$) we deduce the second inequality.
\end{proof} 
Let us now consider an impulsive dynamical system $(X,\varphi,D,I)$. Observe that if  $I(D)\cap D=\emptyset$, then one can find $\eta>0$ small enough satisfying $I(D)\cap \overline{D}_{\eta}=\emptyset$.  Define $\tau^*:X\to \mathbb{R}^+\cup\{\infty\}$   by
$$\tau^*(x)=\left\{\begin{array}{lll}\tau_1(x) & \mbox{ if }& x\notin D,\\
0 & \mbox{ if } & x\in D.
\end{array}\right.$$

\begin{lema}\label{C}
\begin{upshape} Let $(X,\varphi, D,I)$ be an impulsive dynamical system. Assume that $I(D)\cap D=\emptyset$,  $D_{\xi}$ is open for some $\xi>0$ and $\varphi_{\xi}(D_{\xi})\subset X_{\xi}$. Then, we have the following properties:
\begin{enumerate}
\item For all $x \in X_{\xi}$ and $t>0$ such that $\varphi_t(x)\in D_{\xi}$, there exists $\tau<t$ such that $\varphi_{\tau}(x)\in D$.\label{CC1}
\item $\tau^*$ is continuous on $X_\xi \cup D$.\label{CC2}
\item  If $\phi$ is the  semiflow associated to $(X,\varphi,D,I)$, then for all $t\geq 0$ we have $\phi_t(X_{\xi})\subset X_{\xi}$ . \label{CC3}
\end{enumerate}
\end{upshape}
\end{lema}

\begin{proof}
\begin{upshape}
First we prove (\ref{CC1}). Let  $x\notin D_{\xi}$ and $t>0$ such that  $\varphi_t(x)\in D_{\xi}$. Take
\begin{eqnarray*}\label{inf}
\tau=\inf\{s<t: \varphi_s(x)\in D_{\xi}\}.
\end{eqnarray*}
Then, we have   $\varphi_{\tau}(x)\in \partial{D_{\xi}}$ because $D_{\xi}$ is open. Since $D$ is a compact set and $\varphi$ is continuous, there are sequences $(z_n)_{n\geq1}$ in $D$ and $(u_n)_{n\geq1}$ in $(0,\xi)$  such that $z_n\to z\in D$, $u_n\to u \in [0,\xi]$,  and 
$$\varphi_{u_n}(z_n)\to\varphi_u(z)=\varphi_{\tau}(x).$$
If $u=0$, we are done. Assume that $u\neq0$. Then $u\in (0,\xi)$ or $u=\xi$. If $u\in (0,\xi)$, then $\varphi_{\tau}(x)=\varphi_u(z)\in D_{\xi}$. Since $D_{\xi}$ is an open set and the semiflow $\varphi$ is continuous, there exists $\tilde{\tau}<\tau$ such that $\varphi_{\tilde{\tau}}\in D_{\xi}$. This contradicts the definition of $\tau$. On the other hand, if $u=\xi$, then there exists  $\eta\in (0,\xi)$ such that $\varphi_{\tau+\eta}(x)=\varphi_{\xi+\eta}(z)= \varphi_{\xi}(\varphi_{\eta}(z)) \in D_{\xi}$, contradicting the fact that $\varphi_{\xi}(D_{\xi})\subset X_{\xi}$. This proves  (\ref{CC1}).

Now, we prove item (\ref{CC2}). By Theorem 2.7 in \cite{i5} $\tau_1$ is lower semicontinuous on $X\setminus D$. Since $\tau^*(x) =0$ for $x \in D$, we conclude that $\tau^*$ is lower semicontinuous on $X$. Hence, it is enough to prove that $\tau^*$ is upper semicontinuous on $ X_{\xi}\cup D$.
First, let us consider $x \in X_{\xi}$. If $\tau^*(x)=\infty$  then we are done. Assume $\tau^*(x)<\infty$ and fix $\varepsilon>0$. Without loss of generality we can assume  that $\varphi_{\tau^*(x)+\varepsilon}(x)$ belongs  $D_{\xi}$. Since $D_{\xi}$ is an open set, there exists $\gamma>0$ such that 
$$B(\varphi_{\tau^*(x)+\varepsilon}(x),\gamma)\subset D_{\xi}.$$
 Since $\varphi$ is continuous, there exists $\delta>0$ such that for all $ y \in B(x,\delta)$
$$ d(\varphi_{\tau*(x) + \varepsilon} (x),\varphi_{\tau*(x) + \varepsilon}(y))<\gamma.$$  By item (1), for all $y\in X_{\xi}$, there exists $\tau<\tau^*(x)+\varepsilon$ such that  $\varphi_{\tau}(y)\in D$. This implies $\tau^*(y)\leq \tau<\tau^*(x)+\varepsilon$ and $\tau^*$ is upper semicontinuous at $x$. Finally, for the point $x\in D$, there exists $\delta>0$ such that
for all $y\in B(x,\delta) $ we have  $\varphi_{\varepsilon}(y)\in D_{\xi}$. If $y\in D$, then $\tau^*(y)=0$. If $y\in X_{\xi}$, then the above argument shows that $\tau^*(y)<\varepsilon$. Hence, we conclude (\ref{CC2}).

Finally, we prove  (\ref{CC3}).  Let us consider $x\in X_{\xi}$ and suppose that there exists $t\geq 0$ such that $\phi_t(x)\notin X_{\xi}$. This implies $\phi_t(x) \in D$ or $\phi_t(x)\in D_{\xi}$. If $\phi_t(x)\in D$, then by definition of impulsive semiflow we must have $t=0$ and $x\in D$, which is a contradiction. Now, assume that $\phi_t(x)\in D_{\xi}$. By  (\ref{CC1}), there exists $\tau<t$ such that  $\phi_{\tau}(x)\in D$,  hence $\tau=0$  and $x\in D$, which is again a contradiction. This completes the proof of (\ref{CC3}).

\end{upshape}
\end{proof}

For the impulsive dynamical system $(X,\varphi,D,I)$, let us consider the quotient space $X/\sim$, where $\sim$ is the equivalence relation given by 
$$x\sim y \,\,\ \Leftrightarrow \,\,\ x=y, \,\,\ y=I(x), \,\,\ x=I(y) \,\,\ \text{or} \,\,\ I(x)=I(y).$$
Let $\pi:X\to X/\sim$ be the canonical projection and let us write $\pi(x)=\tilde{x}$ for any $x\in X$. We endow $X/\sim$ with the quotient topology. Moreover, we define on $X/\sim$  the pseudometric 
$$\tilde{d}(\tilde{x},\tilde{y})=\inf\{d(p_1,q_1)+d(p_2,q_2)+\dots+d(p_n,q_n)\},$$
where the infimum is taken over all pairs of finite sequences $(p_1,p_2,\dots ,p_n)$ and $(q_1, q_2, \dots, q_n)$ with $p_1\in \tilde{x}$, $q_n\in \tilde{y}$ and $q_i\in \tilde{p}_{i+1}$, for all $i=1,\dots,n-1$.

Recall that the quotient topology is $T_0$ when the equivalence classes are closed. Moreover, this  topology contains the topology induced by the above pseudometric. We will prove that, under certain assumptions, the above pseudometric is actually a metric. We start with the following result.
\begin{lema}\label{cero}
\begin{upshape}
Let $(X,\varphi,D,I)$ be an impulsive dynamical system. If $I(D)\cap D=\emptyset$ and $I$ is Lipschitz, then given $\tilde{x},\tilde{y}\in \pi(D)$, we have $$\tilde{d}(\tilde{x},\tilde{y})=0 \Rightarrow \tilde{x}=\tilde{y}.$$
\end{upshape}
\end{lema}

\begin{proof}

Since $I(D)\cap D=\emptyset$, $I$ is continuous and $D$ is compact, there exists $\alpha>0$ such that for $z\in D$ and $w\in I(D)$, we have $d(z,w)>\alpha$. Take $\tilde{x},\tilde{y}\in \pi(D)$ with $\tilde{d}(\tilde{x},\tilde{y})=0$. For any $0<\varepsilon<\alpha$ there exist $(p_1,p_2,\dots,p_n)$ and $(q_1,q_2,\dots,q_n)$ with $p_1\in\tilde{x}$, $q_n\in\tilde{y}$ and $q_i\in\tilde{p}_{i+1}$ for all $i\in\{1,\cdots,n-1\}$ such that
$$\sum_{i=1}^n d(p_i,q_i)<\varepsilon.$$
Without loss of generality we can assume that $x,y\in I(D)$. Then, we have
$$\tilde{x}=\{x\}\cup I^{-1}(\{x\}) \mbox{ and } \tilde{y}=\{y\}\cup I^{-1}(\{y\}).$$
Define $i_0 =0$ and
$$i_1=\max\{i\in\{1,\dots,n-1\}:  p_j=q_{j-1} \mbox{ for all }j \mbox{ with }0< j\leq i\}.$$
We have  two cases.

\begin{cas2} $p_1\in D$\\
\begin{upshape}
Note that for any $z\notin D\cup I(D)$, $\tilde{z}=\{z\}$. This implies that $q_{i_1}\in D\cup I(D)$. We claim that $q_{i_1}\in D$. Indeed, if $q_{i_1}\in I(D)$ then
$$\alpha  <d(p_{1},q_{i_1})\leq \sum _{i=1}^{i_1} d(p_{i},q_{i})  \leq \sum _{i=1}^n d(p_{i},q_{i})<\varepsilon.$$
This gives is a contradiction.

\end{upshape}
\end{cas2}

\begin{cas2}
\begin{upshape}
$p_1\in I(D)$.\\
Using the same argument as above, one deduces $q_{i_1}\in I(D)$.

\end{upshape}
\end{cas2}
Now, suppose that  $i_1=n$. Accordingly to the cases  $p_1 \in D$ or $p_1 \in I(D)$ respectively,  we have 

$$d(x,y)  =d(I(p_1),I(q_n))\leq C\sum_{i=1}^n d(p_i,q_i)<C\varepsilon$$
or
$$d(x,y)  =d(p_1,q_n)\leq \sum_{i=1}^n d(p_i,q_i)<\varepsilon,$$ where $C>0$ is a  Lipschitz constant for $I$. We conclude $$ d(x,y) \leq \max \{C, 1\} \varepsilon.$$
On the other hand, if $i_1<n$, we must consider two cases: If $p_1 \in D$ then $d(I(p_1),I(q_{i_1})) \leq C\sum_{i=1}^{i_1} d(p_i,q_i) $,  because 
$$   d(I(p_1),I(q_{i_1})) \leq Cd(p_1,q_{i_1})  \leq  C\sum_{i=1}^{i_1} d(p_i,q_i).$$
And, if $p_1 \in I(D)$ then $d(p_1,q_{i_1})\leq \sum_{i=1}^{i_1} d(p_i,q_i)$. Now, suppose  we have defined  $i_{k}$.  Then, we define 
$$i_{k+1}=\max\{i\in\{i_{k}+1,\dots,n-1\}: p_j=q_{j-1} \mbox{ for all } j \mbox{ with } i_k < j\leq i \}.$$
Again,  by using the same argument  as above, we have $q_{i_{k+1}}\in D$  when $p_{i_{k}+1}\in D$ or $q_{i_{k+1}}\in I(D)$ when $p_{i_k+1}\in I(D)$. Moreover, we have 
$  d(I(p_{i_k +1}),I(q_{i_{k+1}})) \leq C\sum_{i=i_k +1 }^{i_{k+1}} d(p_i,q_i) $ when $q_{i_{k+1}} \in D$ and  $ d(p_{i_k +1},q_{i_{k+1}})\leq \sum_{i=i_k +1}^{i_{k+1}} d(p_i,q_i) $ when $q_{i_{k+1}} \in I(D)$. Note that  for any $k$,   $p_{i_k+1}\neq q_{i_k}$, hence we have three alternatives: $q_{i_{k}}=I(p_{i_k+1})$ or $I(q_{i_{k}})=I(p_{i_k+1})$ or $I(q_{i_{k}})=p_{i_k+1}$. Furthermore, there exists $l$ such that $i_l =n$. By using this decomposition we deduce
$$  d(x,y) \leq \sum_{k=0}^{l-1} A_k d(p_{i_{k}+1},q_{i_{k+1}}),$$
where $A_k = 1$ when $p_{i_k +1} \in I(D)$ and $A_k = C$ when   $p_{i_k +1} \in D$. Putting all together we get $$ d(x,y) \leq \max \{ C, 1\} \sum_{i=1}^n d(p_i,q_i)<\max \{ C, 1\}\varepsilon.  $$
Finally, since $\varepsilon$ is arbitrary, we conclude $x=y$ and so $\tilde{x}=\tilde{y}$.

\end{proof}

\begin{propo}\label{metricadelcociente}
\begin{upshape}Let $(X,\varphi,D,I)$ be an impulsive dynamical system. If $I(D)\cap D=\emptyset$ and $I$ is Lipschitz,
then $(\pi(X),\tilde{d})$  is a compact metric space.
\end{upshape}
\end{propo}
\begin{proof}
\begin{upshape}
Since $X$ is a compact metric space and $\pi:(X,d)\to (\pi(X),\tilde{d})$ is continuous and surjective, we conclude that $(\pi(X),\tilde{d})$ is also compact. We claim that for all $\tilde{x},\tilde{y}\in \pi(X)$ with $\tilde{d}(\tilde{x},\tilde{y})=0$ we have $\tilde{x}=\tilde{y}$. Indeed, suppose $\tilde{x}\notin\pi(D)$. Then $x \notin D \cup I(D)$. Put $\alpha = d(x, D \cup I(D))$ and take $0<\varepsilon<\alpha$ . Since $\tilde{d}(\tilde{x},\tilde{y})=0$  there exist $(p_1,p_2,\dots,p_n)$ and $(q_1,\dots,q_n)$ with $p_1=x$, $q_n\in\tilde{y}$ and $q_i\in\tilde{p}_{i+1}$ for all $i\in\{1,\dots,n-1\}$, such that
$$\sum_{i=1}^n d(p_i,q_i)<\varepsilon.$$
Let us consider
$$i_1=\max\{i\in\{1,\dots,n-1\}: p_i=q_{i-1}\}.$$
We have $q_{i_1}\in D\cup I(D)$. But 
$$\varepsilon > \sum_{i=1}^{n} d(p_i,q_i)\geq \sum_{i=1}^{i_1} d(p_i,q_i)\geq d(x,q_{i_1})\geq d(x,D\cup I(D))=\alpha,$$
which gives a contradiction. Therefore we must have $\tilde{x}\in \pi(D)$. By symmetry we also have $\tilde{y}\in \pi(D)$. By Lemma \ref{cero}, we obtain the desired claim. This finishes the proof.
\end{upshape}
\end{proof}
Assuming that $I(D)\cap D=\emptyset$, $I$ is Lipschitz and $D_{\xi}$ is open, we conclude that 
$$\pi(X_{\xi})=\pi(X \setminus D_{\xi})$$
is a compact metric space. Moreover, for any $x,y\in X_{\xi}$ we have $x\sim y$ if only if $x=y$. This shows that $\pi|_{X_{\xi}}$ is a continuous bijection (not necessarily a homeomorphism) from $X_{\xi}$ onto $\pi(X_{\xi})$. We define the induced  semiflow 
$$
\tilde{\phi}:\mathbb{R}_0^+\times\pi(X_{\xi})\to \pi(X_{\xi})
$$
by
$$\tilde{\phi}(t,\pi(x))=\pi(\phi(t,x)).$$

\begin{propo}
\begin{upshape} Let $(X,\varphi,D,I)$ be an impulsive dynamical system. If $I(D)\cap D=\emptyset$, $D_{\xi}$ is open  for some $\xi>0$ and $\varphi_{\xi}(D_{\xi})\subset X_{\xi}$, then $\tilde{\phi}:\mathbb{R}_0^+\times\pi(X_{\xi})\to \pi(X_{\xi})$ is a continuous semiflow.
\end{upshape}
\end{propo}
\begin{proof}
\begin{upshape}
By Lemma \ref{C}, we have that $\tau^*$ is continuous on $X_{\xi}$. Applying Proposition 4.3 in \cite{v2} we obtain the continuity of the semiflow $\tilde{\phi}$.
\end{upshape}
\end{proof}

For the following lemmas, let us consider  $(X,\varphi,D,I)$ a regular impulsive system. Since  $D$ is compact, $I$ is continuous and $ I(D)\cap D =\emptyset$,  there exists $\epsilon>0$ such that the  $\epsilon$-neighborhoods 
$$V_1=\{x\in X: d(x,D)<\epsilon\}, V_2=\{x\in X:d(x,I(D))<\epsilon\}$$ 
of $D$ and $I(D)$, respectively, are disjoint, i.e. $V_1\cap V_2=\emptyset$.

\begin{lema}\label{s1}
\begin{upshape} Let us consider $\phi:\mathbb{R}^{+}_0\times X\to X$  the semiflow  associated to the regular impulsive system  $(X,\varphi,D,I)$. Then, there exist a compact metric space $Y$, a continuous semiflow $\tilde{\phi}:\mathbb{R}_0^+\times Y\to Y$ and a uniform continuous bijection $H:X_{\xi}\to Y$ such that 
\begin{equation}\label{sconj}
\tilde{\phi}_t\circ H=H\circ\phi_t, \mbox{ for all }t\geq 0
\end{equation}
Moreover,
 $\hat{h}_s(\phi|_{X_{\xi}})= h_{top}(\tilde{\phi}).$
\end{upshape}
\end{lema}
\begin{proof}
Let us put $Y=\pi(X_{\xi})$. By properties of  regular impulsive systems we know that $Y$ is a compact metric space. Choose
$H= \pi|_{X_{\xi}} $ and $\tilde{\phi}$ the induced semiflow. These cleary satisfy \eqref{sconj}.
Now, by Lemma \ref{semiconj} and Theorem \ref{A} we have
$$ \hat{h}_s(\phi|_{X_{\xi}})\geq \hat{h}_s(\tilde{\phi})= h_{top}(\tilde{\phi}).$$
In order to prove the other inequality, let us consider $\delta>0$,  $0<\varepsilon<\epsilon < \delta$ , $T>0$ and $E$ a $(\phi,T,\varepsilon,\delta)$-separated set. For all $x,y\in E$ there exists $t_0\in[0,T]$ such that for all $s\in[t_0,t_0+\delta]$
$$d(\phi_{s}(x),\phi_{s}(y))\geq \varepsilon.$$
We claim that $\pi(E)$ is a $(\tilde{\phi},T,\varepsilon)$-separated set. Indeed, suppose that  $\pi(E)$ is not $(\tilde{\phi},T,\varepsilon)$-separated. Then. there exist $x, y\in E$, $x\neq y$ such that for all $t\in[0,T]$, we have
$$\tilde{d}(\tilde{\phi}_{t}(\pi(x)),\tilde{\phi}_{t}(\pi(y)))< \varepsilon.$$
In particular
\begin{eqnarray}\label{contradiccion1}\tilde{d}(\pi(\tilde{\phi}_{t_0}(x)),\pi(\tilde{\phi}_{t_0}(y)))< \varepsilon.\end{eqnarray}
Since $V_1\cap V_2=\emptyset$, we have three different cases.

\begin{cas3}\begin{upshape}
$\phi_{t_0}(x)\notin V_1\cup V_2$.\\
For all $\theta>0$, there exist $(p_1,\dots,p_n)$ and $(q_1,\dots,q_n)$ with $p_1\in \tilde{\phi}_{t_0}(\pi(x))$, $q_n\in \tilde{\phi}_{t_0}(\pi(y))$  and $q_i\in \tilde{p}_{i+1}$ for all $i\in\{1,\dots,n-1\}$, such that 
$$\tilde{d}(\tilde{\phi}_{t_0}(\pi(x)),\tilde{\phi}_{t_0}(\pi(y))) \geq \sum_{i=1}^n d(p_i,q_i)-\theta.$$
If there exists $i\in\{1,\cdots,n\}$ such that $q_i \in D\cup I(D)$, let us consider 
$$l=\min\{i \leq n: q_{i} \in D\cup I(D)\}.$$
Then 
$$\tilde{d}(\pi(\tilde{\phi}_{t_0}(x)),\pi(\tilde{\phi}_{t_0}(y))) \geq \sum_{i=1}^l d(p_i,q_i)-\theta>d(p_1,q_l)-\theta> d(x,q_l)- \theta>\epsilon - \theta.$$
Hence $\varepsilon>\epsilon - \theta$ by  (\ref{contradiccion1}), which gives a contradiction (because $\theta$ is arbitrarily small). On the other hand, if $q_i\notin D\cup I(D)$ for all $i\in\{1,\cdots,n\}$, then 
$$\tilde{d}(\pi(\tilde{\phi}_{t_0}(x)),\pi(\tilde{\phi}_{t_0}(y)))=d(\phi_{t_0}(x),\phi_{t_0}(y))>\varepsilon,$$
which contradicts  (\ref{contradiccion1}).
\end{upshape}\end{cas3}

\begin{cas3}\begin{upshape}
$\phi_{t_0}(y)\notin V_1\cup V_2$.\\
This follows from  the previous case by symmetry.
\end{upshape}\end{cas3}
\begin{cas2}\begin{upshape}$\phi_{t_0}(x),\phi_{t_0}(y)\in V_1\cup V_2$.\\
For all $s\in[t_0,t_0+\delta]$ we have
$$d(\phi_{s}(x),\phi_{s}(y))\geq \varepsilon.$$
By (2) and (3) in Definition \ref{RIDS}, we can choose $s$ such that $\phi_{s}(x)\notin V_1\cup V_2$ or $\phi_{s}(y)\notin V_1\cup V_2$. Applying the previous cases changing $t_0$ by $s$ we get a contradiction.
\end{upshape}\end{cas2}
We conclude
$$|E|=|\pi(E)|\leq s(\tilde{\phi},T,\varepsilon),$$
hence $\hat{h}_s(\phi|_{X_{\xi}})\leq h_{top}(\tilde{\phi})$. This proves the desired result.
\end{proof}

\begin{lema}\label{L1}
\begin{upshape} Let us consider $\phi:\mathbb{R}^{+}_0\times X\to X$  the semiflow  associated to the regular impulsive system  $(X,\varphi,D,I)$. Then
$$\hat{h}_s(\phi)=\hat{h}_s(\phi|_{X_{\xi}})$$
\end{upshape}
\end{lema}
\begin{proof}
\begin{upshape}
Since  
$X_{\xi}\subset X$, we have
$$\hat{h}_s(\phi)\leq\hat{h}_s(\phi|_{X_{\xi}}).$$
In order to prove the other inequality, let us consider $\delta>0$, $\varepsilon>0$, $T>0$ and $E\subset X$ a $(\phi,T,\varepsilon,\delta)$-separated set. Consider the decomposition
$E=E_1\cup E_2$, where  $$E_1=E\cap \overline{D_{\xi}} \mbox{ and } E_2=E\cap \overline{ D_{\xi}}^c.$$
Since $E_2$ is $(\phi|_{\overline{D_{\xi}}^c},T,\varepsilon,\delta)$-separated and $\overline{D_{\xi}}^c\subset X_{\xi}$, we have
$$|E_2|\leq \hat{s}(\phi|_{X_{\xi}},T,\varepsilon,\delta).$$
We claim that there exists $n=n(\varepsilon)>0$ such that
$$|E_1|\leq n\hat{s}(\phi|_{X_{\xi}},T,\varepsilon, \delta).$$
Indeed, since $\overline{D_{\xi}}$ is compact, we can choose $r>0$ such that for all $x,y\in \overline{D_{\xi}}$ 
$$d(x,y)<2r \Rightarrow d(\phi_t(x),\phi_t(y))<\varepsilon, \mbox{ for all } t\in[0,\eta-\xi],$$
where $\eta$ is given in Definition \ref{RIDS}. By compactness, there exists $\{z_k\}_{k=1}^n$ such that
$$\overline{D_{\xi}}\subset \bigcup_{k=1}^n B(z_k,r).$$
Let us consider a choice function  $e:E_1\to \{1,\dots,n\}$ such that $x\in B(z_{e(x)},r)$ for all $x\in E_1$. Then $$E_1=\bigcup_{i=1}^n E_1^i\mbox{, where }E_1^{i}=\{x\in E_1: e(x)=i\}.$$
Therefore, for $x,y\in E_1^{i}$,  we have
$$d(\phi_t(x),\phi_t(y))<\varepsilon, \mbox{ for all } t\in[0,\eta-\xi].$$
If we assume that $\delta<\frac{\eta-\xi}{2}$, then 
$$d_{\delta}^{\phi}(\phi_t(x),\phi_t(y))<\varepsilon, \mbox{ for all }t \in \left[0,\frac{\eta-\xi}{2}\right].$$
On the other hand, since $E$ is $(\phi,T,\varepsilon,\delta)$-separated, for all $x, y\in E_1^{i}$, $x\neq y$ , there exists $t\in[\frac{\eta-\xi}{2},T]$ such that
$$d_{\delta}^{\phi}(\phi_t(x),\phi_t(y))\geq\varepsilon .$$
Since $t\geq\frac{\eta-\xi}{2}>\xi$ and $\varphi_{\xi}(\overline{D_{\xi}})\subset X_{\xi}$, we have that $\phi_{\xi}(E_1^{i})$ is $(\phi|_{X_{\xi}},T-\xi,\varepsilon,\delta)$-separated. Taking all this into account, we obtain 
$$|E_1|\leq n \hat{s}(\phi|_{X_{\xi}},T,\varepsilon, \delta).$$
This proves the claim. Finally, since
$$|E|\leq (n+1) \hat{s}(\phi|_{X_{\xi}},T,\varepsilon, \delta),$$
we have 
$$\frac{1}{T}\log \hat{s}(\phi,T,\varepsilon,\delta)\leq \frac{1} {T} \log (n+1)+\frac{1}{T}\log \hat{s}(\phi|_{X_{\xi}},T,\varepsilon, \delta).$$
Taking limits, we get
$$\hat{h}_s(\phi)\leq \hat{h}_s(\phi|_{X_{\xi}}).$$
This completes the proof of the Lemma.
\end{upshape}
\end{proof}
\begin{proof} [Proof of Theorem \ref{F}]
By Lemma \ref{s1}, and its proof, we can choose $Y=\pi(X_{\xi}) $, $\tilde{\phi}$ the induced semiflow   and $
H= \pi|_{X_{\xi}}$.  Now, by Lemma \ref{s1}, we have 
$$ \hat{h}_s(\phi|_{X_{\xi}})= h_{top}(\tilde{\phi}).$$ %
By  Lemma \ref{L1}, we have 
$$\hat{h}_{\text{s}}(\phi)= \hat{h}_{\text{s}}(\phi|_{X_{\xi}}).$$ 
Putting all together, we get $\hat{h}_{\text{s}}(\phi)= h_{top}(\tilde{\phi})$. This completes the proof.
\end{proof}

\section{An Example}\label{6}
In this section, we show a simple example of a discontinuous semiflow where Bowen's entropies are infinite.  Moreover, we calculate our extended topological entropies. 

Consider the phase space $X$ as the annulus
$$X=\{(r\cos\theta, r\sin\theta)\in \mathbb{R}^2 : 1\leq r\leq 2 ,\theta\in[0,2\pi)\},$$
and define
$$\begin{array}{lccl}\varphi:&\mathbb{R}_0^+\times X&\to& X\\
                          & (t,x) &\to &(r\cos(\theta+t),r\sin(\theta+t))\end{array} $$
for $x=(r\cos\theta, r\sin\theta)$. Equivalently, $\varphi$ is the continuous semiflow of the vector field in $X$ given in polar coordinates by
$$
\left\{\begin{array}{ll}
r'&=0\\
\theta'&=1.
\end{array}\right.$$
Note that the trajectories of $\varphi$ are circles spinning around zero counterclockwise. Put
$$D=\left\{\left(\frac{3}{2},0\right)\right\}\subset X,$$
and define $I:D\to X$ by 
$$I\left(\frac{3}{2},0\right)=\left(1,\pi\right).$$
We associate to $(X,\varphi,D,I)$ the discontinuous semiflow $\phi:\mathbb{R}^+\times X\to X$ given in polar coordinates by 
$$\phi_t(r,\theta)=\varphi_t(r,\theta)$$
if $(r,\theta)\neq (\frac{3}{2},\theta)$, and by
$$\phi_t\left(\frac{3}{2},\theta\right)=\left\{\begin{array}{ll} \varphi_t(\frac{3}{2},\theta) & \textit{ if } t\in[0,2\pi-\theta),\\
\varphi_{t-(2\pi-\theta)}(1,\pi) & \textit{ if } t\geq 2\pi-\theta.
\end{array}\right.$$
Since  $\varphi$ is a continuous semiflow, for all $\varepsilon>0$  there exists $\beta>0$ such that for all $t\geq0$ and $(r,\theta)\in X$ 
$$s\in(t,t+\beta) \Rightarrow d(\varphi_s(r,\theta),\varphi_t(r,\theta))<\varepsilon.$$
We claim that for any fixed $\theta_0\in [0,2\pi)$, $T\geq 2\pi-\theta_0$ and $\varepsilon<\frac{1}{2}$ the set
$$E= \left \{ \phi_t \left( \frac{3}{2},\theta_0 \right) : t\in[0,2\pi-\theta_0) \right\}$$ 
is a $(\phi,T,\varepsilon)$-separated set. Indeed, let us consider $x, y\in E$ with $x\neq y$. There exists $t\in[0,2\pi-\theta_0)$ such that
$$d(\phi_t(x),\phi_t(y))\geq\frac{1}{2}.$$
Hence, for $T\geq 2\pi-\theta_0$ and $\varepsilon<\frac{1}{2}$, $E$ is a $(\phi,T,\varepsilon)$-separated set.  Therefore
$$s(\phi,T,\varepsilon)\geq |E|=\infty \mbox{ and }h_s(\phi)=\infty.$$
On the other hand, for $T> 2\pi-\theta_0$ and $\varepsilon<\frac{1}{4}$, we claim that any $(\phi,T,\varepsilon)$-spanning set $F$ contains $E$. Indeed, if $x\in E$ there exists $y\in F$ such that for all $t\in [0,T]$
$$d(\phi_t(x),\phi_t(y))<\varepsilon.$$
If $y= (r,\theta)$, $r\neq\frac{3}{2}$, then it is easy to see that there exists $t\in [0,T]$ such that
$$d(\phi_t(x),\phi_t(y))>\frac{1}{4}.$$
On the other hand, if  $y= (\frac{3}{2},\theta)$, $x\neq y$,  then there exists $t\in [0,T]$ such that
$$d(\phi_t(x),\phi_t(y))\geq \frac{1}{2}.$$
So $x=y$ and the claim is proved. Therefore $|F|\geq \infty$ which implies
$$r(\phi,T,\varepsilon)=\infty \mbox{ and }h_r(\phi)=\infty.$$

Now, we calculate $\hat{h}_{r}(\phi)$ and $\hat{h}_{s}(\phi)$. Note that in this example we can apply Lemma \ref{B1} which gives
$$\hat{h}_{r}(\phi)\leq \hat{h}_{s}(\phi).$$
Moreover,  $\hat{h}_{s}(\phi)=0$. Indeed, let $T>0$, $\varepsilon>0$ and $\delta>0$, and let $E$ be a $(\phi,T,\varepsilon,\delta)$-separated. We write $E=F\cup E_{\frac{3}{2}}$, where $E_{\frac{3}{2}}=\{(\frac{3}{2},\theta):\theta\in[0,2\pi]\}\cap E$ and $F=E\setminus E_\frac{3}{2}$. Then
$$\begin{array}{rcl}|F|\leq \frac{4\pi}{\varepsilon^2} &\mbox{and} &|E_\frac{3}{2}|\leq \frac{4\pi}{\varepsilon}+1.\end{array}$$
Hence,
$$|E|\leq \hat{s}(\phi,T,\varepsilon,\delta)\leq \frac{8\pi}{\varepsilon^2}.$$
Therefore $\hat{h}_s(\phi)=0$ as claimed. This proves that $\hat{h}_{r}(\phi)=\hat{h}_r(\phi)=0$.

\section*{Acknowledgements}

Nelda Jaque thanks the  Instituto de Matem\'atica of the Universidade Federal do Rio de Janeiro and the Instituto de Matem\'aticas of the Pontificia Universidad Cat\'olica de Valpara\'iso for their kind hospitality during the preparation of this work. The results in this paper are part of her PhD thesis at the Universidad Cat\'olica del Norte.

\end{document}